\documentclass[12pt,reqno]{amsart}

\usepackage{mathrsfs}
\usepackage{multicol}
\usepackage{fancyvrb}
\usepackage{amssymb,latexsym,amsbsy,amsthm,
amsxtra,amscd,amsfonts,amsthm,amsmath,verbatim}
\usepackage[all,cmtip]{xy}
\usepackage{lmodern}
\usepackage[T1]{fontenc}

\usepackage[margin=1in]{geometry}

\usepackage{paralist}

\input xy
\xyoption{all}
\usepackage[colorlinks,pagebackref=true]{hyperref}

\usepackage{mathrsfs}
\usepackage{multicol}
\usepackage{fancyvrb}
\usepackage{color}

\newcommand{\ncom}{\newcommand}
\ncom{\bq}{\begin{equation}}
\ncom{\eq}{\end{equation}}
\ncom{\beqn}{\begin{eqnarray*}}
\ncom{\eeqn}{\end{eqnarray*}}
\ncom{\beq}{\begin{eqnarray}}
\ncom{\eeq}{\end{eqnarray}}
\ncom{\been}{\begin{enumerate}}
\ncom{\eeen}{\end{enumerate}}
\ncom{\olin}{\overline}
\ncom{\f}{\frac}
\ncom{\rar}{\rightarrow}
\def\a{{\bf a}}

\makeatletter

\@addtoreset{equation}{section}
\def\theequation{\thesection.\@arabic \c@equation}

\def\@citecolor{blue}
\def\@linkcolor{blue}
\def\@urlcolor{blue}

\makeatother

\usepackage{amssymb,latexsym,color,
amsxtra,amscd,amsfonts,amsthm,amsmath,verbatim,epsfig}

\usepackage{multicol}
\usepackage{fancyvrb}
\usepackage[colorlinks,pagebackref=true]{hyperref}
\def\@citecolor{blue}
\def\@linkcolor{blue}
\def\@urlcolor{blue}

\def\theequation{\arabic{equation}}
\def\theequation{\thesection.\arabic{equation}}
\numberwithin{equation}{section}

\def\Ass{\operatorname{Ass}}

\def\charac{\operatorname{char}}

\def\depth{\operatorname{depth}}
\def\dim{\operatorname{dim}}

\def\height{\operatorname{ht}}

\def\charac{\operatorname{char}}

\def\minass{\operatorname{MinAss}}
\def\minass{\operatorname{MinAss}}

\newcommand{\kk}{\Bbbk}

\def\lrar{{\longrightarrow}}

\def\A{\mathbb A}
\def\C{\mathbb C}

\def\Gr{\mathcal G}

\def\N{\mathbb N}
\def\P{\mathbb P}
\def\R{\mathcal R}
\def\Z{\mathbb Z}

\newcommand{\pp}{\mathtt p}
\def\af{\mathfrak a}

\def\m{\mathfrak m}
\def\n{\mathfrak n}
\def\p{\mathfrak p}

\newcommand{\ov}{\overline}

\newcommand{\lm}{{\lambda}}

\theoremstyle{plain}
\newtheorem{theorem}[equation]{Theorem}

\newtheorem{proposition}[equation]{Proposition}
\newtheorem{lemma}[equation]{Lemma}

\theoremstyle{definition}

\newtheorem{example}[equation]{Example}

\ncom{\bib}{\bibitem}

\ncom{\limns}{\underset{\underset{s}{\longrightarrow}}{\lim}}
\ncom{\limnr}{\underset{\underset{r}{\longrightarrow}}{\lim}}
\ncom{\maxi}{\underset{\underset{}{i}}{\max}}

\ncom{\limm}{\underset{{ n \to \infty}}{\lim}}
\ncom{\Tprime}{T^{\prime}}
\ncom{\mprime}{\m^{\prime}}

\begin{document}
\title[ ] {Symbolic Rees algebras and \\ [2mm] set-theoretic complete intersections} 
 \author[Clare D'Cruz]{Clare D'Cruz$^*$}
\address{Chennai Mathematical Institute, Plot H1 SIPCOT IT Park, Siruseri, 
Kelambakkam 603103, Tamil Nadu, 
India
} 
\email{clare@cmi.ac.in} 

\author[Mousumi Mandal] {Mousumi Mandal$^\dag$}
\address{Department of Mathematics, Indian Institute of Technology Kharagpur, 721302, India}
\email{mousumi@maths.iitkgp.ac.in}

\author{J. K. Verma}
\address{Department of Mathematics, Indian Institute of Technology, Mumbai, 400076, India}
\email{jkv@iitb.ac.in}

\keywords{Symbolic  Rees algebra, set-theoretic complete intersections, Fermat ideal, Jacobian ideal, edge ideal}
\thanks{$^*$ partially supported by a grant from Infosys Foundation}
 \thanks{$^\dag$ Supported by SERB(DST) grant No.: $\mbox{MTR}/2020/000429$, India}
 \subjclass[2010]{Primary: 13A30, 1305, 13H15, 13P10} 
 
 \begin{abstract}
 In this paper we extend a result of Cowsik on set-theoretic complete intersection  and a result  Huneke, Morales and Goto and Nishida  about Noetherian symbolic Rees algebras of ideals. As applications, we show that the symbolic Rees algebras of the following ideals are Noetherian and the ideals are set-theoretic complete intersections: (a) the edge ideal of a complete graph, (b)  the Fermat ideal  and (c) the Jacobian ideal of a certain hyperplane arrangement.
 \end{abstract}
 \maketitle
 \section{Introduction}

  Throughout this paper $(R, \m)$ is a Noetherian local ring of positive dimension $d$ with  infinite residue field and $ \a \subset R$  is an ideal.
  We say that $\a$  is defined set-theoretically by $s$ elements if there exists $f_1, \ldots, f_s \in \m$ such that $\sqrt{\a} = \sqrt{(f_1, \ldots, f_s)}$. If $s = h:= \height(
  \a)$ and $\sqrt{\a} = \sqrt{(f_1, \ldots, f_h)}$ then  $\a$ is called set-theoretic complete intersection.

One of the remarkable results on set-theoretic complete intersection was proved by   R. C. Cowsik and M. V. Nori.  They showed   that if the characteristic of a field $\kk$, $\charac(\kk) = p >0$, then  the defining ideal of any affine curve in $\A_{\kk}^n$ is a set-theoretic complete intersection (\cite[Theorem~1]{cowsik-nori}).
J. Herzog  proved  that all 
monomial curves in the affine space $\A^3_{\kk}$  over any field $\kk$ are set-theoretic complete intersections. 
This  was  also proved by H. Bresinsky \cite{bresinsky-1979}. 
It has been a long standing open problem to characterise    ideals that   are set-theoretic complete intersections.  We shall review a few of these results.

Symbolic powers of ideals have  played an important role in connection with the problem of set-theoretic complete intersections. The $n$-th symbolic power of an ideal $\a$ is defined as
 ${\displaystyle \a^{(n)}=\bigcap_{\p\in \Ass(\a)}(\a^nR_{\p}\cap R)}$.  
 It has also been of importance because  of its connection with algebraic geometry which  goes back to the work of  O. Zariski \cite{zariski} and M. Nagata  \cite{nagata}.  If $\kk$ is an algebraically closed  field, then the  $n$-th symbolic power of a given prime ideal consists of  functions  that vanish up to order $n$ on the algebraic  variety defined by the prime ideal \cite{eisenbud}. 

Let ${\displaystyle \R_s(\a):= \bigoplus_{n \geq 0} \a^{(n)} t^n}$ denote the symbolic Rees algebra of an ideal $\a$. The study of the symbolic Rees algebras of ideals in a Noetherian local domain was of interest due to a result of D. Rees which gave a counter-example to a problem of Zariski using symbolic Rees algebras.
Rees showed that   if $\p$ is a prime ideal  with $\height(\p)=1$ in a Noetherian normal local domain   of dimension two \cite{rees} then the symbolic Rees algebra is Noetherian if and only if some symbolic power of $\p$ is principal.
Rees' work was generalized  by A. Ooishi  \cite{ooishi-1985}. The above results were studied under the condition $\dim (R/ \a)=1$. In 1990, Goto. et. al studied the Noetherian property of any unmixed ideal in a Noetherian local ring \cite{goto}. 
The symbolic Rees algebra of prime ideals has also been studied in \cite{huneke-1982},  \cite{roberts-1985}, \cite{katz-ratliff-1986}, \cite{schenzel},  \cite{schenzel-1988}, \cite{sannai-2019} \cite{grifo-2021}.

In $1981$ R. C. Cowsik proved  the following remarkable   result  that rejuvenated  interest in symbolic powers of an ideal:
 \begin{theorem}
 \cite{cowsik-1981} Let 
$(R, \m)$ be a Noetherian  local ring of dimension $d$ and $\a \not =\m $ a radical ideal. If  the symbolic Rees algebra 
$\R_{s}(\a) := \oplus_{n \geq 0} \a^{(n)}$ is Noetherian, then $\a$
 is the radical of an ideal generated by $d-1$ elements. In particular, if $\a$ has height $d-1$, then it is a set-theoretic complete intersection.
\end{theorem}
The converse of Cowik's result  need not be true \cite[Corollary~1.2]{goto-nis-wat}. One of the main results in this paper is to  generalize  Cowsik's result on set-theoretic complete intersections. 
Cowsik's   work motivated several researchers to investigate the Noetherian property of the symbolic Rees algebra. In 1987,  C. Huneke gave necessary and sufficient conditions for 
$\R_{s}(\p)$ to be Noetherian when $R$ is a $3$-dimensional regular local ring and $\p$ is a height two prime ideal \cite{huneke-1987}. Huneke's result  was 
generalised  for $\dim R \geq 3$ by M. Morales  in \cite{morales-1991}  and later  by S. Goto \cite{goto}. We state the most general form of the result.
Recall that a local ring $R$ is called unmixed if all the associated primes of the completion of $R$  have the same dimension.

\begin{theorem}
\cite[Theorem~1.1]{goto}
Let $(R, \m)$ be a  local ring of dimension $d$ and $\p$ a prime ideal of height $d-1$. 
\been
\item
Suppose $\R_s(\p)$ is Noetherian. Then  there exists  a system of parameters $x, f_1, \ldots, f_{d-1}$ such that 
 elements $f_i \in \p^{(k_i)}$, $i=1, \ldots, d-1$ with $k_i>0$, satisfying 
\beq
\label{goto condition}
e_{(x, f_1, \ldots, f_{d-1}) }(R)
=    e(R_{\p}) \lm \left( \f{R}{ (\p,x) }\right)   \prod_{i=1}^{d-1} k_i  .
\eeq
 \item
If $R$ is an unmixed local ring containing   a system of parameters $x, f_1, \ldots, f_{d-1}$ satisfying (\ref{goto condition}), then $\R_s(\p)$ is Noetherian.
\eeen
\end{theorem}
These  results provided  a useful way to find  Noetherian symbolic Rees algebras and to find ideals that are set-theoretic complete  intersections.

In this paper we extend the work of Huneke \cite{huneke-1987}, Morales \cite{morales-1991} and Goto and Nishida   \cite{goto} for any arbitrary   ideal $\a$ such that $1\leq \height \a \leq d-1$ (Theorem~\ref{main theorem}, Theorem~\ref{main theorem converse}). We give necessary and sufficient conditions for  the Noetherian property of $\R_s(\a)$. For a radical 
 ideal $\a$, we prove in Theorem \ref{analytic}   that if $\R_s(\a)$ is Noetherian and $R/\a^{(n)}$ is Cohen-Macaulay for large  $n$, then $\a$ is set theoretic complete intersection. 
We demonstrate our results with a few examples. In  Section~3, using our main theorem  we prove  that the symbolic Rees algebra of the  edge ideal of  complete graph is Noetherian. In Section~4, we give  a new proof for the Noetherian property of the symbolic Rees algebra of the Fermat ideal. In section 5, we prove that the symbolic Rees algebra of the Jacobian ideal the defining polynomial of certain hyperplane arrangement  is Noetherian. 

\section{The Main Result}
In this section  we prove the main result of this paper.
We  generalise a  result of  Cowsik,  Huneke, Morales and Goto to  any arbitrary  ideal $\a$ of  height $h$, where $1\leq h\leq d-1$ and give necessary and sufficient condition for the Noetherian property of $\R_s(\a)$. In order to prove the main theorem first we prove few preliminary results.  For an ideal $\a$, the associated graded ring is defined as 
${\displaystyle \Gr(\a) := \bigoplus_{n \geq 0} \a^n / \a^{n+1}}$ and for  any element $x \in \a^n\setminus \a^{n+1} $, let $x^{\star}$ denote the  image  of $x$ in 
$\a^{n}/\a^{n+1}.$

\begin{lemma}
\label{regular}  
Let $(R,\m)$ be a  Noetherian  local ring of dimension $d\geq 1$ and $\a$ be an  ideal with $\dim R/\a=s\geq 1$.
Let $R/\a^{n}$ be Cohen-Macaulay for all $n\geq 1$. 
\been
\item
\label{regular-1}
Then there exist $x_1,\ldots,  x_{s}$ such that  $x_1^{\star},\ldots ,x_{s}^{\star}$ is a $\Gr(\a)$-regular sequence
and 
$x_1,\ldots ,x_{s}$ is an $R$-regular sequence and

\item
\label{regular-2}
$\dim R/\a =\dim R-\height \a $.
\eeen
\end{lemma}
\begin{proof}
(\ref{regular-1}) We prove (a)   by induction on $\dim (R/\a)=s$. First assume that $\dim (R/\a)=1$, then we show that there exists $x_1 \in R$ such that $x_1^{\star}$ is  a $\Gr(\a)$-regular element  and $x_1$ is  an $R$- regular element. 
As $R/\a^n$ is Cohen-Macaulay for all $n\geq 1,$ $\Ass (R/\a^n)=
\{P_1,\ldots ,P_t\}$ where $P_1, P_2,\ldots, P_t$ are the minimal primes of $\a.$

By prime avoidance lemma choose $x\in \m\setminus \displaystyle{\cup_{i=1}^t P_i}$. Then $x$ is $R/\a^{n}$-regular for all $n\geq 1$. To show $x$ is $\Gr(\a)$-regular, it is enough to show that $x$ is regular on  $\a^n/ \a^{n+1}$ for all $n \geq 1$.   From the exact sequence
\begin{equation*}
0\longrightarrow \a^{n}/\a^{n+1}\longrightarrow R/\a^{n+1}\longrightarrow R/\a^{n}\longrightarrow 0 \hspace{.2in}  (n \geq 1),
\end{equation*}
 $x$ is $\a^{n}/\a^{n+1}$-regular for all $n \geq 0$ and hence $x^*$ is $\Gr(\a)$-regular.
We  now show that $x$ is $R$-regular.  There exists $n\geq 0$ so that  $x\in \a^n\setminus\a^{n+1}.$ If there is a $y\in R$ such that  $xy=0$ and $y\neq 0$ then  $y\in \a^m\setminus \a^{m+1} $ for some  $m.$ As $(xy)^* =x^*y^*=0,$ $y^*=0$ which means $y\in \a^{m+1}$ which is a contradiction. Hence $x$ is $R$-regular.
Now assume that 
$\dim R/\a >1$.  
 Using the above argument  choose $x_1$ so that $x_1^{\star}$ is  $\Gr(\a)$-regular and $x_1$ is  $R$-regular.  As $x_1$ is $R/\a^n$- regular for all $n\geq 1$, we have $(x_1)\cap \a^n=x_1 \a^n$ for all $n\geq 1$ which implies that $\Gr(\a\ov R)\cong \Gr(\a)/x_1^* \Gr(\a)$, where $\overline{R}=R/(x_1)$. Since $x_1$ is $R/ \a^n$-regular for all $n \geq 1$, we have $  \overline {R}/  \overline{\a^n}$ is Cohen-Macaulay of dimension $s-1$ for all $n \geq 1$.  By  induction hypothesis,  we can choose $x_2,\ldots, x_{s} \in R$ such that $\ov{x_2}^*,\ldots,  \ov{x_{s}}^*$  is $\Gr(\a\ov R)\cong \Gr(\a)/x_1^* \Gr(\a)$-regular sequence and   $\ov{x_2},\ldots, \ov{x_{s}}$ is $\ov{R}$-regular sequence. Thus  $x_1^*,\ldots ,x_{s}^*$ is a $\Gr(\a)$-regular sequence and 
$x_1,\ldots ,x_{s}$ is a an $R$-regular sequence. This proves (\ref{regular-1}).
 
(\ref{regular-2}) By part (1), since there exists an $R/\a$-regular sequence  $ \underline {x} = x_1,  \ldots, x_s,$ $\dim (   R/ \a + (\underline{x})) =0.$ Hence
 $\overline{\a}$ is an $\m/(\underline{x})$-primary ideal in $\ov{R}=R/ (\underline{x})$.  Since $\underline{x}$ is an $R$- regular sequence,
 $$        \height(\a) 
 \geq \height(\overline{\a}) 
 = \dim (R/ (\underline{x}) )
 = \dim(R) -\dim R/\a =\height(\a).$$
 Therefore  $s=\dim R/\a=\dim R-\height \a$.
\end{proof}

\begin{example} Lemma~\ref{regular} need not hold true if we drop the assumption that $R/\a^n$ is Cohen-Macaulay for all $n \geq 1$.
Let $R= k[[x,y,z]] /(xy,xz)= k[[\overline{x}, \overline{y}, \overline{z}]]$ and $I = (\overline{x})$. Then $R/I$ is Cohen-Macaulay, but $R/I^n$ is not Cohen-Macaulay for all $n \geq 2$.  Since $\dim(R/I) =2$ and  $\depth(R) =1,$ we cannot find a regular sequence of length two in $R$. \qed
\end{example}
 In \cite[Theorem~3.3]{ghnv}, Goto et. al proved  that if $R$ is an unmixed local ring and $\a$ an ideal such that $\ell(\a^{(k)}) = \height(\a^{(k)})$,  for some $k \geq 1$, then $\R_s(\a)$  is Noetherian where  $\ell(\a)$ denote the analytic spread of  $\a$.  With the  additional assumption that $R/\a^{(n)}$ is Cohen-Macaulay for all $n \gg 0$, they proved the converse \cite[Theorem 3.6]{ghnv}.   In the next lemma we give an alternate proof of the converse.
 
\begin{theorem}\label{analytic}
Let $(R,\m)$ be a local ring of dimension $d,$ and $R/\m$ is infinite. Let $\a$ be an ideal of positive height $h$. If   $R/\a^{(n)}$ is Cohen-Macaulay for all  $n \gg 0$ and  $\R_s(\a)$ is Noetherian, then  $\ell(\a^{(k)})=\height(\a^{(k)})$ for some $k\geq 1$ and $\a$ is a set-theoretic complete intersection.
\end{theorem}

\begin{proof}
Since $\R_s(\a)$  is Noetherian,  by \cite[Theorem~3.2]{ghnv} there exists  $m \geq 1$ such that $\a^{(m)r}=\a^{(mr)}$ for all $r \geq 1$. By our assumption,  there exists an $n_0$ such that for all $n \geq n_0$, $R/\a^{(n)}$ is Cohen-Macaulay. Choose $k \geq \{m, n_0 \}$ such that both the assumptions are satisfied. Put $I = \a^{(k)}$. Then 
 for all $n \geq 1$, $R/I^n$ is Cohen-Macaulay. 
Hence by Lemma \ref{regular}(\ref{regular-2}),  $\dim R/I=\dim R-\height I$. Since $R/I^n$ is Cohen-Macaulay we have $\depth (R/I^n)=\dim R/I^n=d- \height(I)$.  By Burch's inequality \cite{burch},  we get 
 $$
 \ell(I) \leq d - \inf_n \depth (R/I^n)=\height(I).
 $$
  As  $\height(I) \leq  \ell(I)$, equality holds which implies that  $\ell(\a^{(k)})=\height(\a^{(k)})$.
 Since $R/\m$ is infinite, there is a  minimal reduction $J=(f_1,\ldots ,f_h)$ of $\a^{(k)}$.  Hence $\sqrt{\a} =\sqrt{\a^{(k)}}=\sqrt{J}$.
\end{proof}

The rest of this section is dedicated for a generalization of results in \cite{huneke-1987} and \cite{morales-1991} and \cite{goto}.  We first prove some preliminary results.

\begin{theorem}
\label{main theorem}
Let $(R,\m)$ be a 
a local ring of dimension $d\geq 1$ with infinite residue field $k.$   Let $\a$ be an ideal of height $h$ with $1 \leq h \leq d-1$
 such that $R/\a^{(m)}$ is Cohen-Macaulay for large $m$ and  $\R_s(\a)$ is Noetherian. Then there exist 
 $x_1,\ldots, x_{d-h},$ a system of parameters for $R/\a,$ and $f_i\in {\a}^{(k_i)}$ for $i=1,2,\ldots, h$ such that 
\begin{enumerate}
\item
\label{main theorem one}
$x_1,\ldots ,x_{d-h}$ is a regular sequence in $R$.
\item 
\label{main theorem two}
$x_1,\ldots ,x_{d-h},f_1,\ldots ,f_h$ is a system of parameters for $R$. 
\item
\label{main theorem three}
Let  $\minass (R/\a):=\{P_1,\ldots ,P_s\}$. Then 
\beqn
e_{(x_1,\ldots ,x_{d-h},f_1,\ldots ,f_h)}(R)=\left ( \displaystyle{\prod_{j=1}^{h}k_j}\right )\displaystyle{\sum_{i=1}^se_{ \a R_{P_i}}(R_{P_i})e_{(x_1,\ldots ,x_{d-h})}(R/P_i)}. 
\eeqn
\end{enumerate}
 \end{theorem}

\begin{proof}
(\ref{main theorem one}) As $\R_s(\a)$ is  Noetherian,  by Theorem \ref{analytic} there exists  $r\geq 1$ such that $\ell( \a^{(r)})=h$.   By our assumption there exists $n_0$ such that $R/ \a^{(n)}$ is Cohen-Macaulay for all $n \geq n_0$. Choose $k =\max\{r,n_0 \}$ and put $I= \a^{(k)}$. Then by the proof of  Lemma~\ref{regular},  we can choose an $R$-regular sequence  $x_1, \ldots ,x_{d-h}\in \m$.

(\ref{main theorem two}) As  $k$ is infinite and $\ell(I) = h, $ we can choose a minimal reduction $J=(f_1,\ldots ,f_{h})R$  of $I$.  Since $\dim R/I=d-h$, by Lemma~\ref{regular},  we can choose  $x_1,\ldots ,x_{d-h}\in \m$  whose images in $R/I$ form a system of parameters for $R/I$.  As  $J$ is a minimal reduction of $I$ we have $\minass(R/I)=\minass (R/J)$, which implies that  $x_1,\ldots ,x_{d-h}$ is  also a system of parameters for $R/J$. Thus $x_1,\ldots ,x_{d-h}, f_1,\ldots, f_h$ is a system of parameters of $R$.
 
 (\ref{main theorem three}) Put $B=R/(\underline{x})$ where $\underline{x} = (x_1, \ldots, x_{d-h})$. Then $JB$ is a  reduction of $IB$. 
 Since $R/I^{n}$ Cohen-Macaulay and $\underline{x}$ is a system of parameters in $R/I^n$ we have 
 \beq
\label{eq3}\nonumber
         e_{({\underline{x}, J)}}(R)=e_J(B)
 &=& e_I(B) \hspace*{5cm} \mbox{ (as $J$ is a reduction of $I$)}\\
&=&        {\limm}  \left[\f{h!}{n^{h}}\lm_B \left( \f{B}{I^{n}B} \right)\right]\\ \nonumber
&=&    {\limm} \left[ \f{h!}{n^{h}}   
 \lm_R \left( \f{R}{(\underline{x})+I^{n}} \right) \right]\\\nonumber
&=&     {\limm}  \f{h!}{n^{h}}     
\left[ e_{\underline{x}R} \left (\f{R}{I^{n}} \right) \right]\\ \nonumber
&=&   {\limm}  \f{h!}{n^{h}}     \left[\displaystyle{\sum_{i=1}^s
        \lm_{R_{P_i}}\left(\f{R_{P_i}}{I^{n}R_{P_i}}\right)
        e_{\underline{x}}\left(\f{R}{P_i}\right)}\right]
        \hspace{.2in} \mbox{(by \cite[Theorem 14.7]{mat})}\\  \nonumber
&=&  {\limm}  \f{k^hh!}{(kn)^{h}}    
\left[ \sum_{i=1}^s \lm_{R_{P_i}} \left(\f{R_{P_i}}{{\a}^{kn}} \right)
          e_{\underline{x}} \left(\f{R}{P_i} \right) \right] \hspace{.2in} 
          \mbox{[since ${\a}^{(kn)}R_{P_j} = {\a}^{kn}R_{P_j}$]}\\
 &=&  k^h \left[ \displaystyle{\sum_{i=1}^s e_{{\a}R_{P_i}}(R_{P_i})
                                     e_{\underline{x}}(R/P_i)}   \right]      .
\eeq
 \end{proof}

 In order to prove the converse of the above theorem we need the following result proved by E. B\"oger.
 An ideal $I$ is said to be equimultiple ideal if $\height I=\ell(I)$.

\begin{theorem}[\cite{B\"oger}] \label{boger}
Let $J$ be an equimultiple ideal of a quasi-unmixed local ring $(R,\frak m)$. Let $I \supseteq J$ be an ideal. Then $J$ is a reduction of $I$ if and only if $e(I_{p}) = e(J_{p})$ for every minimal prime $p$ of $I$ and $\sqrt{I}=\sqrt{J}.$ 
\end{theorem}

\begin{theorem}
\label{main theorem converse}
Let $(R,\m)$ be a quasi-unmixed 
 local ring of dimension $d\geq 1$.  Let $\a$ be an ideal of height $h$ with $1 \leq h \leq d-1$ and $\minass(R/\a)=\{P_1,\ldots ,P_s\}$ such that $\height P_i=h$ for all $i=1,\ldots ,s$.
If there exist $x_1,\ldots, x_{d-h}$ a system of parameters of $R/\a$  and $ f_i\in \a^{(k_i)}$ such that conditions $(\ref{main theorem one})$, $(\ref{main theorem two})$ and $(\ref{main theorem three})$ of  Theorem \ref{main theorem} are satisfied, then $\R_s(\a)$ is Noetherian.
 \end{theorem}
\begin{proof}
By passing to high powers of the $f_i$'s we may assume that $k=k_i$ for all $1\leq i \leq h$. Fix $k$ and put $I={\a}^{(k)}$ and $J=(f_1,\ldots ,f_{h})$. We will show that $J$ is a reduction of $I$. Let $B=R/(x_1,\ldots ,x_{d-h})$. Then as $\underline{x}=(x_1,\ldots ,x_{d-h})$ is $R$-regular, by \cite[Theorem 14.11]{mat}  and condition (\ref{main theorem three}) we have 
\beq
\label{eq7}
e_J(B)
=e_{(\underline{x})+J}(R)
= k^{h} \sum_{i=1}^s
   e_{\a R_{P_i}}(R_{P_i})
   e_{\underline{x}} \left( \f{R}{P_i}\right)  .
\eeq
Let $F=\minass_R (R/J)$. Then for all  $n\geq 1$,  $F=\minass_R (R/J^{n})$. Hence  we get 
\beq
\label{converse eq1}  \nonumber
       {e_J(B)}
&=& \left[\limm \f{h!}{n^{h}} \lm_B(B/J^{n}) \right]\\ \nonumber
& = & \limm \left[ \f{h!}{n^{h}} \lm_R(R/(\underline{x})+J^{n})\right]\\ \nonumber
& \geq & \limm  \f{h!}{n^{h}}  \left[e_{\underline{x}R} \left( \f{R}{J^{n}} \right) \right]\\ \nonumber
& = & \limm \f{h!}{n^{h}}  \left[ \displaystyle{\sum_{P\in F}\lm_{R_P}(R_P/J^{n}R_P)e_{\underline{x}}(R/P)} \right]
\mbox{by \cite[Theorem~14.7]{mat}}\\
&=&{\displaystyle{\sum_{P\in F}e_{JR_P}(R_P)e_{\underline{x}}(R/P)}}\label{converse eq}
\eeq
As $\height I=\height J=h$ and $J \subseteq I$  for all $i=1, \ldots, s $,  $P_i\in  F$,
\beq
\label{converse eq2}
 \sum_{P\in F}e_{JR_P}(R_P)e_{\underline{x}}(R/P)
 \geq  \sum_{i=1}^se_{JR_{P_i}}(R_{P_i})e_{\underline{x}}(R/P_i).
\eeq
Now from (\ref{converse eq}) and  (\ref{converse eq2}) we get 
 \beq\label{eq11}
 {e_J(B)}
\geq \sum_{i=1}^se_{JR_{P_i}}(R_{P_i})e_{\underline{x}}(R/P_i).
\eeq 
Since $J\subseteq I$, for $1\leq i \leq s$,
\beq
e_{JR_{P_i}(R_{P_i})} &\geq & e_{IR_{P_i}}(R_{P_i})
\label{converse eq4}
\eeq
Since $I^{n}R_{P_i}   = ({\a}^{(k)}R_{P_i})^{n}={\a}^{kn}R_{P_i}$,  we have 
\beq
 \label{converse eq5} 
e_{IR_{P_i}}(R_{P_i})= \limm \frac{h!}{n^h}\lm(R_{P_i}/I^{n}R_{P_i})
                                 =   \limm  \frac{h!}{n^h}\lm(R_{P_{i}}/{\a}^{kn}R_{P_i})
                            = k^{h}e_{{\a}R_{P_i}}(R_{P_i}) .  
\eeq 
Thus from (\ref{eq11}), (\ref{converse eq4}), (\ref{converse eq5}) and (\ref{eq7}) we have 
\begin{eqnarray*}
e_J(B)&\geq &\displaystyle{\sum_{i=1}^se_{JR_{P_i}}(R_{P_i})e_{\underline{x}}(R/P_i)} \\
& \geq & \displaystyle{\sum_{i=1}^se_{IR_{P_i}}(R_{P_i})e_{\underline{x}}(R/P_i)}\\
& = & \displaystyle{\sum_{i=1}^se_{{\a}R_{P_i}}(R_{P_i})k^{h}e_{\underline{x}}(R/P_i)}\\
& = & e_{J}(B) \label{eq6}
\end{eqnarray*}
Therefore we get $F=\{P_1,\ldots ,P_s\}$ and $e_{JR_{P_i}}(R_{P_i})=e_{IR_{P_i}}(R_{P_i})$ for all $1\leq i \leq s.$  Now since $R$ is quasi-unmixed and $J$ is an equimultiple ideal, by Theorem~\ref{boger},  $J$ is a reduction of $I$, which implies that $\ell(I)=h$. Thus by Theorem \cite[Theorem 3.6]{ghnv}, $\R_s(\a)$ is Noetherian. 
\end{proof}

\begin{example} Theorem \ref{main theorem converse} may not be true if $R$ is quasi-unmixed but  $\height(\a)=0$. 
Let $R = k[[x,y]]/ (xy)  = k[[ \overline{x}, \overline{y}]]$ and let $\a =  (\overline{x})$.  Then $J = (\overline{x})$ is a minimal reduction of $I$ and $f= \overline{x + y}$ is a  parameter  for $R$ as $\dim(R)=1$.  
Since $\a^m R_{\a} \cap R = \a   =  (\overline{x})$ for all $m \geq 1$, $R/ \a^{(m)}$ is Cohen-Macaulay for all $m \geq 1$, but  $\R_s(\a)$ is not Noetherian. \qed
 \end{example}

\begin{example} Theorem~\ref{main theorem converse} does not hold true if $R$ is not quasi-unmixed.
Let $R = k[[x,y,z]]/ (xy, xz)  = k[[ \overline{x}, \overline{y}, \overline{z}]]$ and let $\a =  (\overline{x},\overline{y})$.  Put  $x_1 =   (\overline{x + y})$. Then $(x_1)  \a = \a^2.$ Hence $(x_1) $  is a minimal reduction of $\a$. Moreover,  if $f_1 = \overline{x + y+z}$, then $x_1, f_1$   is a system of parameters  for $R.$
 Since   $\lm(R/(x_1, f_1)^n)=\binom{n+1}{2}+n,$
$e_{(x_1,f_1)}(R)=1.$ As   $\a$ is a prime ideal, $h=1$, $k_1 = 1$, $s=1$, we get
\beqn
 \left ( \displaystyle{\prod_{j=1}^{h}k_j}\right )\displaystyle{\sum_{i=1}^se_{\a R_{P_i}}(R_{P_i})e_{(x_1,\ldots ,x_{d-h})}(R/P_i)  }
= e_{( \overline{x}, \overline{y})}R_{(x,y)} \cdot e_{(   \overline{x + y+z ) }} (R/ ( \overline{x},\overline{y}))=e_{(x_1,f_1)}(R)=1.
\eeqn
As   $\a^n R_{\a} \cap R =  (\overline{x}, \overline{y^n})$ for all $n \geq 1,$ $\R_s(\a)$ is not Noetherian. \qed
 \end{example}

 \section{The symbolic Rees algebra  of the edge ideal of complete graph}
 
 Let $\kk$ be a field and $S = \kk[x_1, \ldots, x_{n}]$, a polynomial ring in $n$ indeterminates.   Let $G$ be a complete graph on $n$ vertices and $I=I(G)  = (x_i x_j \mid   \{i,j\} \text{ is an edge of } G) \subset S$ be the corresponding edge ideal of $G$. 
  By  a result of Herzog, Hibi and Trung \cite[Corollary~1.3]{herzog-hibi-trung},  $\R_s(I) $ is Noetherian.  We give another proof of this result.
  
 Let ${\bf x} := x_1, \ldots, x_n$ and $\sigma_1({\bf x}),\ldots, \sigma_n({\bf x})$  be the elementary symmetric  functions of  $\bf x.$
  \begin{proposition} \label{sym}  The functions $\sigma_j({\bf x})$ for $j=1, 2, \ldots, n$ form a homogeneous system of parameters of $S$ and
  $$\lm(S/(\sigma_1, \sigma_2, \ldots, \sigma_n))=n!.$$
  
  \end{proposition}
  \begin{proof}
   Note that each $x_j$, for $j=1, 2, \ldots, n$, is a root of the monic polynomial 
  $$x^n-\sigma_1x^{n-1}+\sigma_2x^{n-2}-\cdots+(-1)\sigma_n.$$
  Hence $S$ is an integral extension of the polynomial ring $R=\kk[\sigma_1, \sigma_2, \ldots, \sigma_n].$ Therefore $J=(\sigma_1, \ldots, \sigma_n)S$ is $\m=(x_1, x_2, \ldots, x_n)S$-primary. Thus $\sigma_1, \sigma_2, \ldots, \sigma_n$ is a homogeneous system of parameters of $S.$
  
  (2) Since $\sigma_1, \sigma_2, \ldots, \sigma_n$ is an $S$-regular sequence, the Hilbert series of $S/J$ is given by
  $$
         H(S/J, u)=\sum_{n=0}^\infty \lm((S/J)_n)u^n
  =\frac{(1-u)(1-u^2)\ldots(1-u^n)}{(1-u)^n}=\prod_{j=1}^{n-1}(1+u+u^2+\cdots+u^j).$$
  Therefore $\lm(S/J)=H(S/J, 1)=n!.$
  \end{proof}
  
  \begin{example}
  \label{complete-graph}
Let $S = \kk[x_1, \ldots, x_{n}]$. Put $R= S_{\m}$ and $\af = I(G)R$ where $\m = (x_1, \ldots, x_{n})$. 
 \been
 
 \item
    \label{complete-graph-11}
     $\af$ is  unmixed and $R/ \af^{(k)}$ is  Cohen-Macaulay for all $k \geq 1$.
 \item
   \label{complete-graph-2}
 $\R_s(\af)$ is Noetherian. 
 \item
   \label{complete-graph-3}
 $\af$ is a set-theoretic complete intersection. 
\eeen
\end{example}
\begin{proof}  (\ref{complete-graph-11}) Put $\p_i =  (x_1, \ldots, x_{i-1}, x_{i+1}, \ldots, x_{n}  )$, $1 \leq i \leq n$.  
By  \cite[Corollary 3.35]{tyul} it follows that 
   for all $n \geq 1$, 
 \beqn
 \af^{(k)} = \bigcap_{i=1}^{n} \p_i^{k}. 
  \eeqn
As $\height \p_i=n-1$ for all $\p_i\in \Ass (S/\a)$, $\a$ is an unmixed ideal.  
Since $x=\sigma_1 \not \in \p_i$ for all $i=1, \ldots, n$ it is a nonzerodivisor on $ R/ \af^{(k)}$ for all $k \geq 1$. Hence $R/\a^{(k)}$ is  Cohen-Macaulay for all $k\geq 1$.

(\ref{complete-graph-2}) Let $x=\sigma_1.$ Then for all $j=2, \ldots, n$,   
  \beq
  \label{example-symmetric-fi} 
  \sigma_{j} ({\bf x}) &\in& I^{(j-1)}  \hspace{.5in} \mbox{\cite[Lemma 2.6]{bocci}}\\
\label{example-symmetric-multiplicity-1}
e_{(x)} ( R/ \p_i)
&=&  \lm\left( \f{R} {(x_1, \ldots, x_{i-1}, x, x_{i+1}, \ldots, x_{n}  )}\right)
=1\\
\label{example-symmetric-multiplicity-2} \nonumber
e_{\af}(  R_{\p_i})
&=& \lm \left( \f{R_{\p_i}}{\af R_{\p_i} }\right)
 = \lm \left(  \f{R_{\p_i}}{ \p_iR_{\p_i}}\right) 
=1\\
e_{ (x, \sigma_1, \ldots, \sigma_n)}(R) &=& \lambda \left( \frac{R} {(\sigma_1, \ldots, \sigma_n)R }\right)
= n!.
\eeq

 By the equation  (\ref{example-symmetric-fi}), 
$k_j := j-1$ for all $j=2, \ldots, n$. 
  Hence by Proposition \ref{sym}
  \beq
  \label{example-symmetric-RHS}
  \prod_{j=2}^{n} k_j \sum_{i=1}^{n} e_{\af}(  R_{\p_i})e(x, R/ \p_i)
  = (n-1)! n = n!=e_{(x,\sigma_2, \ldots, \sigma_n)}(R).
  \eeq
 By Proposition \ref{sym}, equation (\ref{example-symmetric-RHS}) and Theorem \ref{main theorem} we conclude that  $\R_s(\af )$ is Noetherian.
 
 (\ref{complete-graph-3}) By (\ref{complete-graph-11}), (\ref{complete-graph-2}) and Theorem \ref{analytic}, $\a$ is a set-theoretic complete intersection.
   \end{proof}

  \section{The symbolic Rees algebra of the Fermat  ideal}
  
 Let   $S = \C[x,y,z]$ be the polynomial ring and $\m = (x,y,z)$.  Let $n \geq 3$ and 
\beqn
r_n := y^n-z^n, \hspace{.2in} s_n =z^n-x^n, \hspace{.2in} t_n =  x^n-y^n, \text{ and } 
    J_n=(x r_n, y s_n, z t_n).
    \eeqn
    The ideal  $J_n$ is called the Fermat ideal. The ideal $J_n$ defines a set of $n^2 + 3$  points in $\P^2$.
 For $n=3$, this ideal was studied in  \cite{dum-sze-gas} in relation with the containment problem. In \cite{nagel-alexandra} the authors study the the symbolic Rees algebra of $J_n$. 
 
\begin{lemma}
\label{lemma-fermat-ideal}
Let $R = S_{\m}$ and $I_n= J_nR$.  Then
\been
\item
 \label{lemma-fermat-ideal-1}
 $J_n$ is a radical ideal.
 \item
 \label{lemma-fermat-ideal-2}
$ e(R/ I_n) = e(S/ J_n) = n^2 + 3$.
 \eeen
\end{lemma}
\begin{proof} Let $\eta$ be a primitive $n^{\rm th}$ root of unity. 
Put $P_{ij} = (y-\eta^i z, z -\eta^j x)$  ($i,j=0, \ldots, n-1$), $P_1 = (y,z)$, $P_2 = (x,z)$, $P_3 = (x,y)$ and 
\beq
\label{ex2-defn-Kn}
K_n =  \bigcap_{i,j=0}^{n-1
} P_{ij} R \bigcap P_1R \bigcap P_2R \bigcap P_3R.
\eeq 
Then  clearly, $I_n  \subseteq  K_n$.
  To prove the lemma it is enough to show that equality holds. 
A minimal free resolution of $J_n $ is 
 \beq 
  \label{minimal free resolution}
  0  
\rightarrow
 S[-n-3] \oplus S[-2n]  
\xrightarrow{ \displaystyle
 \begin{pmatrix}
yz  & x^{n-1} \\
xy & z^{n-1}\\
-xz & -y^{n-1}
\end{pmatrix}}  S[-n-1)]^3  
\xrightarrow{  \displaystyle
 \begin{pmatrix}
xr^n & ys^n & zt^n\\
  \end{pmatrix}} 
 S
\rightarrow \frac{S}{J_n}
 \rightarrow 0.
  \eeq
Since  $\m$ is a   maximal ideal in $S$,  $\depth(S/ J_n)_{\n} \leq  \dim (S/ J_n)_{\m}=1$.   From  (\ref{minimal free resolution}) the projective dimension of $(S/ J_n)_{\m}=2$.
By  the Auslander-Buchsbaum formula, $ (S/ J_n)_{\m}$  is Cohen-Macaulay.
 Using the exact sequence (\ref{minimal free resolution}), we obtain the Hilbert series of $S/J_n$. Let $u$ be an indeterminate. Then 
 $$H(S/J_n, u)=\frac{1-3 u^{n+1}+u^{n+3}+u^{2n}}{(1-u)^3}$$
Since $\dim S/J_n=1,$ we can write $H(S/J_n, u)=\frac{f(u)}{1-u}$ for some $f(u)\in \Z[u].$ Let $g(u)=1-3 u^{n+1}+u^{n+3}+u^{2n}.$ Then $g(u)=(1-u)^2f(u).$  Therefore 
$$e(S/J_n)=f(1)=\lim_{u\to 1} \frac{g(u)}{(1-u)^3}=n^2+3.$$

Every minimal prime of $K_n$ is also that of $J_n.$ Let $a=n^2+3$ and 
$$J_n=   \bigcap_{i,j=0}^{n-1
} Q_{ij} R \bigcap Q_1R \bigcap Q_2R \bigcap Q_3R
     \bigcap_{j=a+1}^b Q_j$$ be a reduced primary decomposition of $J_n$ and let $Q_j$ be $\p_j$-primary for all $j=1,2,\ldots,b.$ By the associativity formula for multiplicities, we have
$$
e(\m, R/I_n)=\sum_{j=1}^be(\m, R/\p_j) \lm ((R/I_n)_{\p_j})=n^2+3.
$$
This shows that $a=b=n^2+3$ and $Q_j=\p_j$ for all $j=1,2,\ldots, a.$ Therefore 
$I_n=K_n$ and $I_n$ is a radical ideal.
\end{proof}

 \begin{example}
 \label{example-fermat-ideal}
Let $R = S_{\m}$ and $I_n= J_nR$.  Then
\been
\item
 \label{example-fermat-ideal-1}
For all $n\geq 3,$ the analytic spread of $I_n$ is $3$.

\item
 \label{example-fermat-ideal-2}
For all $k \geq 1$, $R/I_n^{(k)}$  is Cohen-Macaulay.
\item
 \label{example-fermat-ideal-3}
Put $\a= I_n$.  Then 
\been
\item
 \label{example-fermat-ideal-a}
$\mathcal R_s(\a)$ is Noetherian.
\item
 \label{example-fermat-ideal-b}
$\a$ is a set-theoretic complete intersection.
\eeen
\eeen
\begin{proof}   (\ref{example-fermat-ideal-1})
Suppose that $\ell(I_n)=2.$ Since $I_n$ is a radical ideal and $R_\p$ is regular for all minimal primes of $I_n,$ by \cite[Theorem]{cn1976},
$I_n$ is generated by two elements. This is a contradiction. Hence $\ell(I_n)=3.$ \\
(\ref{example-fermat-ideal-2})
Since $I_nR$ is a height two ideal of a $3$-dimensional local ring,
$I_n^{(k)}$ is an unmixed ideal for all $k\geq 1.$  Hence $R/I_n^{(k)}$ is a Cohen-Macaulay local ring for all $k\geq 1.$

(\ref{example-fermat-ideal-a}) 
      Put $r = r_n$, $s=s_n$, $t=t_n$,  $x_1 = x+y+z$,   $f_1 = rs^{n-2}t - 2 r^{n-2}st$ and
 $f_2 =  rs^{n-2}t + 2( s^{n-2} r^2 x^n + t^{n-2} s^2 y^n  + r^{n-2} t^2 z^n).$ We claim that 
$x_1, f_1, f_2$ is a system of parameters in $R$.  We can write
\beqn
 r s^{n-2}t - 2 r^{n-2}st = rst ( s^{n-3}- 2 r^{n-3} )
\eeqn
Let  $\p$ is a prime ideal and $\p \supseteq (x_1,f_1,f_2)$. 
If $r \in \p$, then  as $r+s+t=0$,
\beqn
\p \supseteq (x_1, r, f_2)
 = (x_1, r,     t^{n-2} s^2 y^n  ) 
 = (x_1, r, (-r-s)^{n-2}s^2y^n) 
 = (x_1, r,   s ^ny^n  ).
\eeqn
Let $\eta$ be a primitive $n^{\rm th}$ root of unity. Then we can write
\beqn
r &=& y^n-z^n = \prod_{i=1}^{n} (y- \eta^i z)\\
s &=& z^n-x^n = \prod_{j=1}^{n} (z- \eta^j x).
\eeqn
Hence, for some $i,j = 1, \ldots, n$,
\beq
\label{mprimary} \nonumber
             \p 
&\supseteq& (x_1, y- \eta^i z, (z- \eta^j x)  y)\\ \nonumber
&=& (x+y+z, y- \eta^i z, (z- \eta^j (x+y+z -y-z)) y)\\ \nonumber
&=& (x+y+z, y- \eta^i z,  (z- \eta^j ( -y-z)) y)\\ \nonumber
&=& (x+y+z, y- \eta^i z,  (z+  \eta^j ( y+z)) y)\\ \nonumber
&=& (x+y+z, y- \eta^i z,  (z+  \eta^j ( \eta^i z +z))  ( \eta^i z))\\ \nonumber
&=& (x+y+z, y- \eta^i z,   \eta^i z^2( (1+   \eta^{i+j}  + \eta^j  )))  \\
&=&  (x+y+z, y- \eta^i z,   z^2).
\eeq
Hence $\p=\m$.
 Similarly, if $s \in \p$ or   $t \in \p$, then 
 using the similar argument as in (\ref{mprimary}) we get  $\p=\m$. 
   
We now compute $e_{(x_1)} (R/ (f_1,f_2))$. Since $x,f_1, f_2$ is a system of parameters, 
\beq
\label{ex2-mult-sop} \nonumber
&&e_{(x_1)} (R/ (f_1,f_2)) \\ \nonumber
&=&  e_{(x_1, f_1, f_2)}(R)\\
&=& \begin{cases}
e_{(x_1 , r, f_2)}(R) + e_{(x_1, s, f_2)}(R) + e_{(x_1, t, f_2)}(R)  & n=3\\
e_{(x_1 , r, f_2)}(R) + e_{(x_1, s, f_2)}(R) + e_{(x_1, t, f_2)}(R) + e_{(x_1, s^{n-3}- 2r^{n-3}, f_2)}(R) & n >3\\
\end{cases}.
\eeq
Let  $n>3$ and $s^{n-3}- 2 r^{n-3} \in \p$.  Then  
 \beqn
 \p 
 &\supseteq&  (x_1,  s^{n-3}- 2 r^{n-3} , f_2) \\
  \eeqn
  
  Let $\zeta$ be the primitive $(n-3)^{\rm th}$ root of unity. Then
  \beqn
   s^{n-3}- 2 r^{n-3} = \prod_{i=0}^{n-3} (s-  \sqrt[n-3]{2}\zeta^i r)
  \eeqn
Fix $i$ ($0 \leq i< n-3$).  Put  $\xi = \zeta^i$ and  $v=  s-ar$ where $a= \sqrt[n-3]{2} \xi r$.  Since   $r + s + t =0$, 
\beq
\label{ex2-simplifying-terms} \nonumber
t &=& -(r+s) = - (1+a) r -v\\ \nonumber
f_2 
&=&   rs^{n-2}t + 2( s^{n-2} r^2 x^n + t^{n-2} s^2 y^n  + r^{n-2} t^2 z^n)\\ \nonumber
&=& r ( v + ar)^{n-2} (-r(1+a) -v)  + 2 ( v + ar)^{n-2} r^2 x^n\\ \nonumber
&&+ 2   ( -(1+a) r + v)^{n-2} (ar)^2 y^n +2 r^{n-2} (  -(1 +a) r -v)^2 z^n \\ \nonumber
&=& -a^{n-2}(1+a) r^n + 2 a^{n-2} x^n r^n  +(-1)^{n-2}2a^2  (1+a)^{n-2} y^n r^n   + 2 (1+a)^2 z^n r^n + g(v) \\
&=&    [a^{n-2}(1+a)  + 2 a^{n-2} x^n  +(-1)^{n-2}2a^2  (1+a)^{n-2} y^n    + 2 (1+a)^2 z^n] r^n + g(v),    
\eeq
where $g(v)$ denotes the terms in $f_2$ which involve  positive powers of $v$.

Since $a^{n-2}(1+a)  + 2 a^{n-2} x^n  +(-1)^{n-2}2a^2  (1+a)^{n-2} y^n    + 2 (1+a)^2 z^n$ is a unit in $R$, 
 if $v \in \p$, then 
$\p \supseteq (x_1, v, f_2) = (x_1, v, r^n) = (x+y+z,  s-  a r, r^n)$. Hence,
  $\p \supseteq  (x+y+z,  s,r)$ which implies that for some $j,k=0, \ldots, n-1$,
  $\p \supseteq (x+y+z, z- \eta^j x, y- \eta^k z)$. Hence $\p = \m $. 

Since  $r+s+t=0$, $(x_1, r, f_2) = (x_1, r, t^{n-2}s^2y^n)$, $x_1, r, f_2$ is a system  of parameters. Hence 
\beq
\label{ex2-x1rf2-sop} \nonumber
 e_{(x_1 , r, f_2)}(R) 
& =& e_{(x_1, r, t^{n-2} s^2 y^n)}(R)\\ \nonumber
& =& e_{(x_1, r, y^n)}(R) + e_{(x_1, r, t^{n-2})}(R) + e_{(x_1, r, s^2)}(R)\\
& =& n^2 + n^2(n-2) + 2n^2
 = n^3 + n^2.
\eeq
Similarly, 
\beq
\label{ex2-sx1sf2-sop}
e(x_1, s, f_2) = e(x_1, t, f_2) = n^3 + n^2. 
\eeq
We now compute $e(x_1, s^{n-3}-2r^{n-3}, f_2)$. Since $x_1, s^{n-3}-2r^{n-3}, f_2$ is a system of parameters, 
\beq
\label{ex2-x2sn-3f2-sop}
e(x_1,s^{n-3}-2r^{n-3}, f_2) 
= n e(x_1, s^{n-3} -2 r^{n-3}, r) = n e( x_1, s^{n-3}, r) =  n^3(n-3).
\eeq 
Combining (\ref{ex2-mult-sop}), (\ref{ex2-x1rf2-sop}), (\ref{ex2-sx1sf2-sop}) and (\ref{ex2-x2sn-3f2-sop}) we get that for all $n \geq 3$
\begin{equation}\label{mul1}
e_{(x_1)}(R/ (f_1,f_2))
= 3(n^3+n^2) + n^4-3n^3 = n^4 + 3n^2= n^2(n^2 + 3).
\end{equation}
As $f_1, f_2 \in I^{(n)}$, $k_1=k_2=n$  and we have
\begin{equation}\label{mul2}
 \left ( \displaystyle{\prod_{j=1}^{2}k_j}\right )
        \displaystyle{\sum_{i=1}^{n^2+3} e_{\a R_{P_i}}(R_{P_i})e_{(x_1)}(R/P_i)} 
        = n^2 (n^2 + 3)
\end{equation}
Then by (\ref{mul1}) and Theorem \ref{main theorem} we have $\mathcal R_s(\a)$ is Noetherian.
        
(\ref{example-fermat-ideal-b})  follows from  (\ref{example-fermat-ideal-2}), (\ref{example-fermat-ideal-a})   and   Theorem~\ref{analytic}.
\end{proof}

\end{example}

\section{The Jacobian ideal of hyperplane arragements}
This example was motivated by the paper by J.~Migliore, U.~Nagel, and H.~Schenck \cite{mig-nag-sch}.

\begin{example}
\label{jacobian}
Let $S = k[x,y,z,w]$ and $f = w(x+y)  (x+y+z+w)$ and $ (J(f))S$, the Jacobian ideal of $f$. Put $R = S_{\m}$ where $\m = (x,y,z,w)$ and ${\a} = (J(f))R$.  
Then 
\been
\item
\label{jacobian-one}
${\a}$ is a height two  unmixed ideal and $R/{\a}^{(n)}$ is Cohen-Macaulay for all $n \geq 1$.

\item
\label{jacobian-two}
The symbolic Rees algebra $\R_s({\a})$ is Noetherian.

\item
\label{jacobian-three}
${\a}$ is a set-theoretic complete intersection.
\eeen

\begin{proof} (\ref{jacobian-one}) One can verify that $f_x=f_y$ and 
\beqn
{\a} &=&  (f_x=2xw+2yw+zw+w^2,  f_z=xw+yw, f_w=x^2+2xy+y^2+xz+yz+2xw+2yw)\\
&=&  (zw+w^2, x^2+2xy+y^2+xz+yz, xw+yw)\\
&=& ( w(z+w), (x+y)(x+y+z), w(x+y))\\
&=& (z+w, x+y) \cap (w, x+y) \cap (w, x+y+z).
\eeqn
Put    $\p_1 = (z+w, x+y)$, $\p_2= (w, x+y)$ and  $\p_3=  (w, x+y+z)$.
Then 
\beqn
{\a}^{(n)} = \p_1^n \cap \p_2^n \cap \p_3^n.
\eeqn
Consider the exact sequences:
\beq
 \label{ses p1 p2}
 0 \lrar \f{R}{\p_1^n \cap \p_2^n}
 \lrar \f{R}{\p_1^n } \oplus \f{R}{ \p_2^n}
 \lrar \f{R}{\p_1^n + \p_2^n}
 \lrar 0,\\
 \label{ses p1 p2 p3}
 0 \lrar \f{R}{\p_1^n  \cap \p_2^n \cap \p_3^n}
 \lrar \f{R}{\p_1^n \cap \p_2^n } \oplus \f{R}{ \p_3^n}
 \lrar \f{R}{ (\p_1^n \cap \p_2^n) + \p_3^n}
 \lrar 0.
 \eeq
 Note that $R/\p_1^n$,  $R/\p_2^n$ and $R/\p_3^n$ are Cohen-Macaulay. 
 Since 
\beqn
 \p_1^n + \p_2^n &=& 
 (   (x+y)^{n-i} (w^i, (z+w)^i):i=0, \ldots, n)
 \eeqn 
 $\height(\p_1^n + \p_2^n)=3$ and $\sqrt{\p_1^n + \p_2^n } = (x+y, z,w)$.  We claim that $x$ is a nonzerodivisor on $R/ \p_1^n + \p_2^n$.
After a change of variables let $y^{\prime} = x+y$, $z^{\prime} = w +z$. Then  we have
 \beqn
 (\p_1^n + \p_2^n)S &=& (   (y^{\prime})^{n-i} (w^i, (z^{\prime})^i):i=0, \ldots, n)  \subseteq k[x, y^{\prime}, z^{\prime},w],\\
 ((\p_1^n \cap \p_2^n) + \p_3^n )S &=& ( (z^{\prime}, y^{\prime})^n \cap (w, y^{\prime})^n) + ( w, y^{\prime} +z^{\prime})^n  \subseteq k[x, y^{\prime}, z^{\prime},w].
 \eeqn
We can write $\p_1^n + \p_2^n = \cap Q_j$ (finite intersection) where each $Q_j$ is generated by pure powers of the variables (see \cite[Theorem~1.3.1]{herzog-hibi}. Then as $(y^{\prime})^n, (z^{\prime})^n, w^n \in \p_1^n + \p_2^n$ and $x$ does not occur in any of generators of $\p_1^n + \p_2^n$, each $Q_j$ is of the form $( (y^{\prime})^{n_{1,j}}, (z^{\prime})^{n_{2,j}}, w^{n_{3,j}})$ for some
 ${n_{1,j}}, {n_{2,j}}{n_{3,j}} \in  \N$. Since $\sqrt{\p_1^n + \p_2^n}S = \cap \sqrt{Q_j}S$ and  $\sqrt{Q_j}S = (y^{\prime}, z^{\prime}, w)$ for all $j$, localizing we conclude that $(y^{\prime}, z^{\prime}, w)R$  is the only associated prime of $(\p_1^n + \p_2^n)R$ and hence $x$ is a nonzerodivisor on $R/ (\p_1^n + \p_2^n)$.  By depth lemma applied to the exact sequence \ref{ses p1 p2}, we see that $R/(\p_1^n \cap \p_2^n)$
 is Cohen-Macaulay.
 Similarly, we conclude that $x$ is a nonzerodivisor on $R/ ( \p_1^n \cap  \p_2^n) + \p_3^n$. By the depth lemma applied to the exact sequence (\ref{ses p1 p2 p3}), it follows that $R/\a^{(n)}=R/(\p_1^n\cap \p_2^n\cap \p_3^n)$ is Cohen-Macaulay.
    
   (\ref{jacobian-two}) Put $x_1 = x$, $x_2 =z$, $g_1 = w(z+w)$, $g_2 = w(x+y)$, $g_3 = (x+y) (x+y+z)$,  $f_1 = g_1 + g_3$ and $f_2 =  f = w(x+y)(x+y+z+w) $. Then $f_1 \in {\a}$. As
   \beqn
   zf_2 &=& zw (x+y) (x+y+z+w) = g_1 g_2 - g_2 g_3,
   \eeqn
 $zf_2\in \a^2 \subseteq {\a}^{(2)} \subset \sqrt{{\a}^{(2)}} = \p_1 \cap \p_2 \cap \p_3$. Since $z \not \in \p_i$, for $i=1,2,3$,  $f_2 \in {\a}^{(2)}$.
Now,
 \beq\label{eq1} \nonumber
         e_{(x_1, x_2, f_1, f_2)}(R)
 &=& e_{(x, z, g_1 + g_3, f)}(R)\\ \nonumber
&=& e_{(x,z,w(z+w) +  (x+y) (x+y+z), w(x+y)(x+y+z+w) )}(R)\\ \nonumber
&=& e_{(x,z, w^2 + y^2, wy(y+w))}(R)\\ \nonumber
& = &e_{( x,z,w^2 + y^2, y)}(R) + e_{( x,z,w^2 + y^2, w)}(R)  + e_{( x,z,w^2 + y^2, y+w)}(R) \\
&=& 6. 
 \eeq
Since $k_1 = 1$ and $k_2=2$, we have
 \beq\label{eq2} \nonumber
 \left ( \displaystyle{\prod_{j=1}^{2}k_j}\right )\displaystyle{\sum_{i=1}^se_{\a R_{P_i}}(R_{P_i})e_{(x_1,x_{2})}(R/P_i)} 
 &=& 1 \cdot 2  \cdot \sum_{i=1}^3 e_{\a R_{P_i}}(R_{P_i})e_{(x_1,x_2)}(R/\p_i)  \\ \nonumber
&=& 2 \cdot \sum_{i=1}^3 e_{{\p_i}{R_{\p_i}}}(R_{\p_i}) e_{(x,z)} (R/ \p_i)\\ \nonumber
 &=& 2 \cdot 3 \\
 &=& 6.
 \eeq
 By (\ref{eq1}), (\ref{eq2}) and Theorem~\ref{main theorem converse}, $R_s(\a)$ is Noetherian.
 
 (\ref{jacobian-three}) Since $R/\a^{(n)}$ is Cohen-Macaulay for $n\geq 1$ and $R_s(\a)$ is Noetherian by Theorem~\ref{analytic}, $\a$ is set theoretic complete intersection.
 \end{proof}

\end{example}

{\bf Acknoledgement:} We thank Prof. S. Goto for many useful discussions during his visit at IIT Bombay.

 \end{document}